\documentclass[11pt,a4paper, DIV9, oneside]{scrartcl}
\pdfoutput=1
\usepackage[english]{babel}
\usepackage[utf8x]{inputenc}

\usepackage[T1]{fontenc}
\usepackage[active]{srcltx}

\usepackage{latexsym}
\usepackage{amsmath}
\usepackage{amsthm}
\usepackage{amsfonts}
\usepackage{amssymb}
\usepackage{psfrag}
\usepackage{graphicx}

\usepackage{natbib}

\usepackage[pdfpagemode=UseOutlines ,plainpages=false
,hypertexnames=false ,pdfpagelabels
,hyperindex=true,colorlinks=true]{hyperref}
	\makeatletter%
	\Hy@breaklinkstrue%
	\makeatother%
\usepackage{color}
\definecolor{darkred}{rgb}{0.6,0.0,0.1}
\definecolor{darkgreen}{rgb}{0,0.5,0}
\definecolor{darkblue}{rgb}{0,0,0.5}
\hypersetup{colorlinks ,linkcolor=darkblue ,filecolor=darkgreen
,urlcolor=darkblue ,citecolor=black}

\usepackage{ifpdf}
\ifpdf
\else
\fi

\usepackage{abbrev-package}
\usepackage{abbrev-copula}

\begin{document}

\makeatletter%
\def\@fnsymbol#1{\ensuremath{\ifcase#1\or * \or \star \or 1 \or 2\or 3\or  , \or
g\or h\or i\else\@ctrerr\fi}}%
\makeatother%

\author{{\sc\large Maik Schwarz}\thanks{Institut de
    statistique, biostatistique et sciences actuarielles, Universit\'e
    catholique de Louvain,
    Belgium. \texttt{ingrid.vankeilegom@uclouvain.be, schwarz@phare.normalesup.org}}
  \and {\sc\large Geurt Jongbloed}\thanks{Delft Institute of Applied Mathematics,
Delft University of Technology, The Netherlands. \texttt{G.Jongbloed@tudelft.nl}} \and {\sc\large Ingrid Van Keilegom}$^\ast$
}

\title{On the identifiability of copulas in bivariate competing risks models}

\renewcommand{\asymp}{\sim}

\date{January 10, 2012}
\maketitle
\vspace{-3em}
\begin{abstract}
\noindent \textbf{Abstract.}\quad
\noindent In competing risks models, the joint distribution of the event times
is not identifiable even when the margins are fully known, which has
been referred to as the ``identifiability crisis in competing risks
analysis'' \citep{Cro:91}. We model the dependence between the event
times by an unknown copula and show that identification is actually
possible within many frequently used families of copulas. The result is then
extended to the case where one margin is unknown.


\end{abstract}

{\footnotesize

}

\renewcommand{\eqref}[1]{(\ref{#1})}
\newcommand{\mb}[1]{\mathbf{#1}}
\newcommand{\mbb}[1]{\mathbb{#1}}
\newcommand{\mt}[1]{\mathrm{#1}}
\newcommand{\rv}{random variable}




\newtheoremstyle{mysc}
  {1em}
  {1em}
  {\it}
  {}
  {\sffamily\bfseries}
  {}
  {.5em}
  {}

\newtheoremstyle{myex}
  {1em}
  {1em}
  {\rm}
  {}
  {\sffamily\bfseries}
  {}
  {.5em}
  {}

\theoremstyle{mysc}
\newtheorem{prop}{Proposition}[section]
\newtheorem{assumption}[prop]{Assumption}
\newtheorem{coro}[prop]{Corollary}
\newtheorem{theo}[prop]{Theorem}
\newtheorem{defin}[prop]{Definition}
\newtheorem{lem}[prop]{Lemma}
\newtheorem{defass}[prop]{Definition and Assumption}
\newtheorem{rem}[prop]{Remark}
\newtheorem{illu}[prop]{Illustration}
\newtheorem{examples}[prop]{Examples}

\theoremstyle{myex}
\newtheorem{example}[prop]{Example}

\numberwithin{equation}{section}

\section{Introduction}
\label{sec:introduction}

The theory of competing risks is concerned with the analysis of
multiple possible causes of a certain event (``risk'') in a system of
interest. As an example, consider an animal experiment in which mice
die either from a disease (event time~$X$) or from the side effects of
some treatment (event time~$Y$). In this setting, one observes only
the minimum of the two event times, corresponding to the time of
death, and a variable indicating the cause of death, that is
\begin{equation}\label{eq:1}
Z:=X\wedge Y \quad\text{ and }\quad \Delta := \1_{X\leq Y}.
\end{equation}
In this paper, we investigate under which conditions the joint
distribution of $(X,Y)$ can be identified based on the
observations. Of course, in practice one could consider many more
potential risks even in this simple example, but in this paper we
treat the bivariate case exclusively. The focus lies in particular on
the case where there is dependence between the two risks~$X$ and $Y$,
which is modelled by a bivariate copula $C:=C_{\et\ct}$. Indeed, the
assumption of independence between $X$ and $Y$ seems unrealistic in a
practical context. In the animal experiment example, one might
suspect a positive correlation between the two risks, because it is
conceivable that the vulnerability to the treatment's side effects
depends on the course of the disease.

The identifiability of joint distributions in competing risks models
is not a recent topic. Already in the 1950s, \cite{Cox:59} pointed
out an identifiability problem in the bivariate case for independent
risks, and \cite{Tsi:75} showed in the general multivariate case
that the joint distribution of the failure times cannot be identified
by their minimum.  \cite{Cro:91} showed that ``the situation is even
worse than previously described'' -- even when the marginal laws of
$X$ and $Y$ are known, the joint distribution function is not
identifiable. This is a very undesirable property of the model, and so
\citeauthor{Cro:91} called out the ``identifiability crisis in competing risks
analysis''.

\citeauthor{Tsi:75}' observations have been the starting point for
investigations on conditions or modifications of the model which allow
for identification and estimation of the event time distribution, and
a variety of such models has been studied over the past decades. In
order to obtain identifiability of the joint distribution of $X$ and
$Y$, one has to restrict the class of possible models. For example,
one can exclude the independent case and restrict the class of
admissible dependence structures. \cite{BG:78} follow this approach
and show identifiability of the bivariate normal and the bivariate
exponential distributions introduced in \cite{MO:67} and
\cite{Gumb:60}. Instead of fixing the precise dependence structure of
the joint distribution in advance, \cite{SR:83} suppose that a certain
hazard ratio involving the event and the censoring time is
known. Under such assumptions, they derive pointwise bounds on the
marginal survivor function. \cite{EMY:03} use still a different
condition, involving a partial derivative of the conditional survivor
function of the event time given that the censoring time is larger
than a given threshold.

Another approach which has been the topic of many research articles
consists in modelling the dependence structure between the event and
the censoring time using a copula. In view of Sklar's theorem, this
kind of model allows for flexibility in the modelling of the
dependence structure without affecting hypotheses on the margins. The
identifiability of the copula and the margins can then be treated in
two separate steps.

\cite{ZK:95} and \cite{RW:01} suppose that the event and censoring
times are dependent via some known copula that is nowhere
constant. They show that their marginal distributions are identifiable
if their support is $(0,\infty)$ and develop the ``copula graphic
estimator''. \cite{KM:88} also work under the hypothesis of a known
copula; more specifically, they choose a Clayton copula with known
parameter.

Of course, supposing the copula to be known is quite as unrealistic as
supposing the risks $X$ and $Y$ to be independent, but in view of the
results in \cite{Cro:91} cited above, identification is impossible if
the copula is completely unknown. It is therefore natural to ask if
the copula can be identified within certain (parametric) classes.  It
is the scope of the present work to answer this question. To this end,
we proceed as follows. 

In Section~\ref{sec:ident-copul-surv}, we assume both of the marginal
distributions to be entirely known and concentrate on the
identification of the copula $C$ which describes the dependence
between the risks $X$ and $Y$. Although the practical relevance of
this setting is limited, its investigation is very instructive. We
develop assumptions on the family of admissible copulas such that
identification is possible based on the joint density of the
observations, and we show that various parametric classes of copulas
do actually satisfy such conditions. In particular, making use of a
recent result by \cite{Wys:12}, we prove that many well-known classes
of Archimedean copulas are identifiable. Certain classes of asymmetric
copulas turn out to be identifiable as well. We also give examples of
symmetric and asymmetric classes of copulas which cannot be
identified.

In Section~\ref{sec:unknown-FX}, we treat the more realistic case
where $F_X$ is completely unknown. Not surprisingly, additional model
assumptions are needed in order to obtain identifiability in this more
general setting. In particular, we assume that~$\Delta$ and $Z$ are
stochastically independent.  In the case where the copula $C_{XY}$ is
known, \cite{BV:08} develop an estimator of $F_X$ even when~$F_Y$ is
unknown. In contrast to this, we show that even if the copula is unknown, $F_X$ can be
identified if $F_Y$ is known. It is obviously impossible to consider all
copulas. However, we will see that the same subclasses of Archimedean
copulas as in Section~\ref{sec:ident-copul-surv} allow for
identification. We will elaborate on the advantages and limitations of
the model in Section~\ref{sec:unknown-FX} and discuss its relation to
the classical Koziol-Green model.

The simultaneous identification of $F_X$ and the copula in more
general models remains an open question. Though the techniques
developed in the present work will hopefully prove helpful in its
solution, the identifiability crisis is not over yet and requires
further investigation, as we finally discuss in
Section~\ref{sec:conclusion}.


\section{Known margins}
\label{sec:ident-copul-surv}

Throughout this section, suppose that the distributions $F_X$ and
$F_Y$ of event and censoring time $\et$ and $\ct$ are known. The
observations $(Z,\Delta)\in\R\times\{0,1\}$ are given by
\eqref{eq:1}. Assuming that the copula $C$ is absolutely continuous with respect to
the Lebesgue measure on the unit square with density $c$ in the sense
that
\[C(x,y)=\int_{-\infty}^x\int_{-\infty}^yc(u,v)\,dvdu,\]
we can write
\begin{multline*}
  \P[X\leq x, Y\leq y] = C (F_X (x), F_Y (y))
= \int_0^{F_X(x)} \int_0^{F_Y(y)} c(u,v) dvdu 
\end{multline*}
Consequently,
\begin{multline*}
  \P[Z>z,\Delta = 1] = \P[Y>X>z] 
= \int_{x=F_X(z)}^1 \int_{y=F_Y(F_X^{-1}(x))}^1 c(x,y) dydx
\end{multline*}
 and we obtain for the joint density $h$ of $(Z,\Delta)$ that
 \begin{align}\label{eq:6}
   \begin{split}
     h(z,0) &= f_Y(z) \int_{x=F_X(z)}^1 c(x,F_Y(z))dx=f_Y(z) - \frac{\partial}{\partial y} C(F_X(z),F_Y(y))\Big|_{y=z}\\
     h(z,1) &= f_X(z) \int_{y=F_Y(z)}^1 c(F_X(z),y)dy=f_X(z) - \frac{\partial}{\partial x} C(F_X(x),F_Y(z))\Big|_{x=z}.
   \end{split}
\end{align}
Note that this implies the basic equality
\begin{equation}
  \label{eq:jointH}
  C(F_X(z),F_Y(z)) = F_X(z) + F_Y(z) - H(z),
\end{equation}
where $H$ denotes the distribution function of $Z$.
From observation~\eqref{eq:6}, one sees that two copulas can be
distinguished based on the distribution of $(Z, \Delta)$ if and only if either of their
partial derivatives do not coincide on the curve 
   \begin{equation}
\Psi=(F_X(t), F_Y(t))_{t\in\R}\label{eq:curve_Psi}.
\end{equation}
 This proves the following result.

\begin{theo}\label{theo:generalident}
  Let $X$ and $Y$ be random variables with differentiable distribution
  functions $F_X$ and $F_Y$, respectively, jointly distributed
  according to an unknown copula belonging to a class $\mathcal{C}$ of
  copulas having a density, and let $(Z,\Delta)$ be the observable
  random variables defined in~\eqref{eq:1}. The class~$\mathcal{C}$ is
  identifiable based on the joint distribution of~$(Z,\Delta)$ if and
  only if for any two different copulas $C,C'\in\mathcal{C}$ there
  exists $z\in\R$ such that
  \begin{multline*}
    \left.\frac{\partial}{\partial x}C(F_X(x),F_Y(z))\right|_{x=z}
    \neq \left.\frac{\partial}{\partial
        x}C'(F_X(x),F_Y(z))\right|_{x=z} \\ \text{or}\quad
    \left.\frac{\partial}{\partial y}C(F_X(z),F_Y(y))\right|_{y=z}
    \neq \left.\frac{\partial}{\partial
        y}C'(F_X(z),F_Y(y))\right|_{y=z}.
  \end{multline*}
\end{theo}
This theorem provides us with a necessary and sufficient
identifiability criterion for arbitrary classes of absolutely
continuous copulas and margins admitting densities.  In practice, this
criterion will be difficult to verify for a given class, though. As
noted, the class of all copulas is far too big to be identifiable as
has been shown by \cite{Cro:91}. In the remainder of the present
section we will therefore restrict our attention to more particular
classes of copulas and develop identification criteria that are easier
to apply.


\subsection{Symmetric copulas}
\label{sec:symmetric-copulas}

It follows from representation~\eqref{eq:6} that whenever $F_X=F_Y$
and in addition both first partial derivatives of~$C$ with respect
to~$x$ and~$y$ coincide along the curve $\Psi$ defined in~\eqref{eq:curve_Psi}, the joint density~$h$
does not depend on~$d$, which happens for instance when the copula is
symmetric. In this case, the curve $\Psi$ is the diagonal of the unit
square, and the density of~$Z$ takes the simpler
form
\begin{equation}\label{eq:4}
h(z) = 2f_Y(z) - 2 \frac{\partial}{\partial y} C(F_Y(z),F_Y(y))|_{y=z} 
\end{equation}
or, equivalently,
\begin{equation}\label{eq:3}
\delta_C(F_Y(z)):=C(F_Y(z),F_Y(z)) = 2F_Y(z) - H(z), 
\end{equation}
where $H$ denotes the distribution function of $Z$ and where $\delta_C$ is
called the diagonal section of $C$.

\begin{theo}\label{thm:symmetric}
Suppose that
the marginal distributions $F_X$ and $F_Y$ are arbitrary but known.
 Then, the class of all symmetric
  copulas is not identifiable.
\end{theo}
 \begin{proof}{}{}
   Consider first the case $F_X=F_Y$. Then the curve 
$\Psi$ defined in~\eqref{eq:curve_Psi}
is the diagonal of
   the unit square.  In order to show the
   non-identifiability of the class of symmetric copulas, we construct two
   distinct symmetric copulas which coincide on a strip of positive
   width around the diagonal and which hence yield the same
   distribution~$H$.  To this end, define first functions from
   $[0,1]^2$ to $\{0,\pm1\}$ according to
\begin{multline*}
  q_{(x,y)}^\eta(s,t) :=
  \operatorname{sgn}[(s-x)(t-y)]\;\I{[(|s-x|\vee|t-y|)<\eta]} \\
  + \operatorname{sgn}[(s-y)(t-x)]\;\I{[(|s-y|\vee|t-x|)<\eta]}.
\end{multline*}
The support of the function $q^\eta_{(x,y)}$ consists of two squares
of side length $2\eta$ centred in $(x,y)$ and $(y,x)$,
respectively. These squares in turn consist of four smaller squares on
which $q^\eta_{(x,y)}$ is constant $+ 1$ or $-1$. Now let $C$ be a
symmetric copula admitting a density $c$ and let $x,y,\eta,\varepsilon\in (0,1)$
be such that $\supp(q^\eta_{(x,y)} )$ is in $[0,1]^2$ and does not
intersect with a neighbourhood of the diagonal as illustrated in Figure~\ref{fig:suppq} and
$c(s,t)>\varepsilon$ for $(s,t)\in\supp(q^\eta_{(x,y)})$. It is then
easy to see that the function $c + \epsilon q^\eta_{(x,y)}$ is a
density yielding a symmetric copula $C'$ which coincides with $C$ on a
band of positive width around the diagonal. Thus, the two
different symmetric copulas $C$ and $C'$ give rise to the same distribution
$H$ according to~\eqref{eq:6}, showing that the class of symmetric
copulas is not identifiable. 

The general case $F_X\neq F_Y$ is straightforward applying the
appropriate transformations.
\end{proof}

\begin{figure}[h]
  \centering 
  \includegraphics[scale=.955]{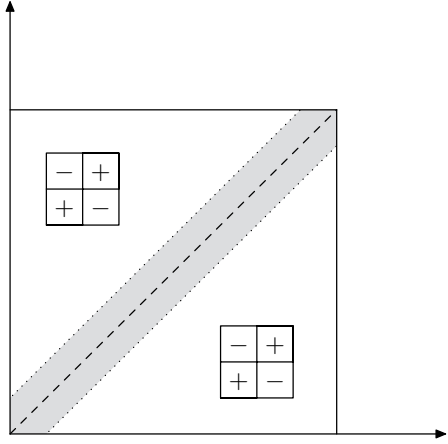}
 \caption{The support of the piecewise constant function
    $q^\eta_{(x,y)}$ consists of two squares that must not intersect
    with a neighbourhood of the diagonal}
  \label{fig:suppq}
\end{figure}


\subsection{Archimedean copulas}
\label{identarchcop}
We have seen that the class of all symmetric copulas is not identifiable. In this
section, we restrict our attention further and consider only
Archimedean copulas, that is copulas of the form
\[C(x,y) = \varphi^{[-1]}(\varphi(x) + \varphi(y))\]
for some generator function $\varphi\in\Omega$, where $\Omega$ denotes the set
of decreasing convex functions from $[0,1]$ to $[0,\infty]$ with
$\varphi(1)=0$. The pseudo-inverse $\varphi^{[-1]}$ of $\varphi$ is
given by
\[
\varphi^{[-1]}(t)=\left\{\begin{array}{cc}\varphi^{-1}(t) & \mbox{ if } 0\le t\le \varphi(0) \\ 0 & \mbox{ if }t>\varphi(0).\end{array} \right.
\]
In the case where $\varphi$ is unbounded, the associated copula is
called \emph{strict} and we have $\varphi^{[-1]}=\varphi^{-1}$. Because of their nice properties, Archimedean copulas are very commonly used in
the modelling of dependence and are therefore an interesting family of
copulas to consider. 

Recent work by \cite{Wys:12} on the construction of copulas from
diagonal sections allows us to state an identification result for a
subfamily of the class of Archimedean copulas in the special case where the
margins are equal. 
\begin{theo}\label{theo:wysocki_ident}
Suppose that
the marginal distributions $F_X$ and $F_Y$ are arbitrary but known. If $F_X = F_Y$, then the
subclass of strict Archimedean copulas generated by 
\[\Omega_1 := \{\varphi\in\Omega\;|\;
\varphi(0)=\infty,\; \lim_{t\uparrow  1}\varphi^\prime(t)<0 \}\]
is identifiable. 
\end{theo}
\begin{proof}{}{}
  The generator of an Archimedean copula being uniquely determined up
  to a multiplicative constant, we may assume that $\lim_{t\uparrow
    1}\varphi^\prime(t) = -1$ without loss of generality. Lemma~1 in
  \cite{Wys:12} states that under this assumption, the diagonal
  sections of two Archimedean copulas coincide (up to a multiplicative
  constant) if and only if their generators do so, that is, if the
  copulas are the same.  As by virtue of~\eqref{eq:3}, the diagonal
  section~$\delta_C$ of the copula can be identified when the margins
  are equal, this implies the identifiability of the whole copula.
\end{proof}
Theorem~\ref{theo:wysocki_ident} implies identifiability for
several well-known classes of Archimedean copulas when both margins
are equal on the unit interval. Important examples of such classes
can be found in \citet[ Table~4.1]{Nel:06}; in particular, the Frank
copulas are identifiable.  It is noteworthy that obviously even the
union of all the identified subclasses of~$\Omega_1$ is identifiable.
But the result is restricted to the special case of equal margins
and to families of strict copulas. For example, the class of strict
Clayton copulas is identifiable by virtue of the theorem, but some
Clayton copulas are not strict ($\theta < 0$). Also note that the theorem cannot
 be applied to every family of strict Archimedean copulas.
For instance, the result does not apply to the Gumbel family, because
$\varphi_\theta^\prime(1)=0$ for all $\theta>1$.

We have seen that Theorem~\ref{theo:wysocki_ident} allows for identification within
quite a large class of strict copulas, but only in the special case of
equal margins. We still need a more general criterion for classes of
strict copulas which are not covered by the theorem as well as for
nonstrict copulas and general margins.

Before stating the next result, let us recall that the level
sets of a copula $C$ are given by $\{(u,v)\in[0,1]^2\;|\; C(u,v) =
t\}$. For an Archimedean copula and for $t>0$, this level set consists
of the points on the level curve $\varphi(u) + \varphi(v) =
\varphi(t)$. For $t=0$, this curve is called zero curve; it is the boundary of the copula's
zero set. All level curves of an Archimedean copula are convex. 

\begin{figure}[h]
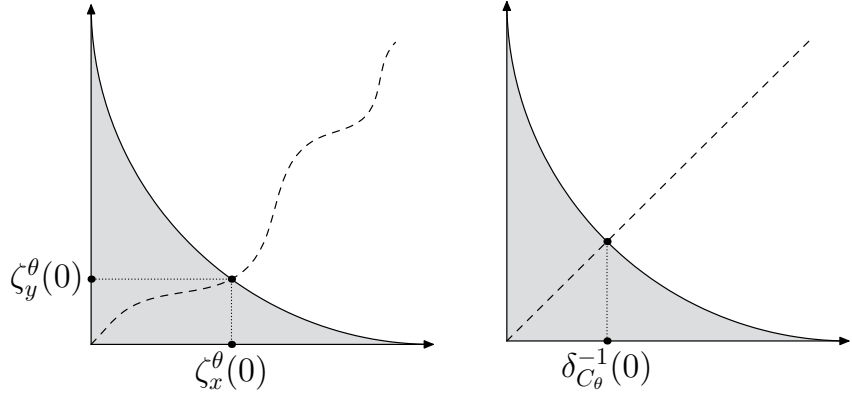

  \centering
  \includegraphics[scale=.9]{zero.2}\hspace{2em}
  \includegraphics[scale=.9]{zero.1}
  \caption{An Archimedean copula's zero curve and its intersection
    with the curve $\Psi=(F_X(t), F_Y(t))_{t\in\R}$ (on the left),
  and the special case $F_X=F_Y$ (on the right). The copula's
  zero set is shaded in grey.}
  \label{fig:zero1}
\end{figure}

\begin{theo}\label{res:one-parameter}Suppose that
the marginal distributions $F_X$ and $F_Y$ are arbitrary but known.
 Let $\Phi:=\{\varphi_\theta \;|\;
  \theta\in\Theta\subset\R\}\subset\Omega$ be a family of generators
  and write $C_{\theta}$ for the Archimedean copula generated by
  $\varphi_{\theta}$.  For $z\in[0,1]$ denote by
  $\zeta^\theta(z)=(\zeta^\theta_x(z),\zeta^\theta_y(z))$ the
  intersection point of the copula $C_\theta$'s level curve of level
  $z$ with the curve $\Psi$ given in~\eqref{eq:curve_Psi}. If for some $z\in[0,1]$ at least one of the two coordinates of
  $\zeta^\theta(z)$ is strictly monotonic as a function from $\Theta$ to
  $[0,1]$, then the class of copulas generated by $\Phi$ is
  identifiable.
\end{theo}
\begin{proof}{}{} Suppose without loss of generality that $\zeta_x^\theta(z)$ is
strictly monotonic in $\theta$. The distribution function $H$ of $(Z,\Delta)$ is
  obviously identifiable. It is thus sufficient to show that the true
  parameter $\theta$ can be computed based on the knowledge of~$H$,
  which we therefore suppose to be known. Then, by virtue
  of~\eqref{eq:jointH} and using that both margins are known, the copula
  can be reconstructed completely along the path~$\Psi$. Consequently,
  the coordinate $\zeta^\theta_x(z)$ can be computed in this case for any
  choice of~$z$. (See Figure~\ref{fig:zero1} on the left for an illustration of
  the case $z=0$.) The strict monotonicity in~$\theta$ ensures
  that~$\zeta^\theta_x(z)$ uniquely characterises~$\theta$ and
  thus~$C_\theta$.
\end{proof}

\begin{example} 
  Let us consider the special case where the marginal distributions
  $F_X$ and $F_Y$ are the same and strictly monotonic on their whole
  support. Let $C$ be an Archimedean copula with generator
  $\varphi\in\Omega$. It is easily seen that we have, for $z\in(0,1]$,
\begin{multline*}
  \delta_C(F_X(x)) = C(F_X(x),F_X(x))=z\iff
  \varphi^{[-1]}(2\varphi(F_X(x)))=z\\\iff\varphi(F_X(x))=\frac12\varphi(z)\iff
  F_X(x)=\varphi^{-1}\Big(\frac12\varphi(z)\Big)=:\delta_C^{-1}(z).
\end{multline*}
For $z=0$, the second equivalence is not true in
general. Nevertheless, the same definition of $\delta^{-1}_C$ can be
used for $z=0$ as well, with the convention that $\delta^{-1}_C(0) =
\sup\{x\in\R\;|\;\delta_C(x) = 0\}$.  The intersection point
$\zeta(z)$ of the copula's level curve of level $z$ with the diagonal
can be expressed in terms of the inverse diagonal section, namely
$\zeta(z)=(\delta^{-1}_C(z),\delta^{-1}_C(z))$.  A parametric class of
Archimedean copulas $\{C_\theta\}$ is thus identifiable if for some
$z$, the inverse diagonal section $\delta^{-1}_{C_\theta}(z)$ is
strictly monotonic as a function of $\theta$.

Choosing $z=1/2$, this condition can be applied to the Gumbel family. Recall
that the Gumbel family could not be treated with
Theorem~\ref{theo:wysocki_ident}.  If $\varphi$ is bounded
(i.e. if~$C$ is nonstrict), $\delta^{-1}_C(0)$ is strictly positive
and corresponds to the intersection of the diagonal with the boundary
of the copula's zero set, see Figure~\ref{fig:zero1} on the
right. Supposing that $\varphi$ belongs to a parametric class,
knowledge of $\delta^{-1}_C(0)$ is often sufficient in order to
characterise~$C$ within the class.
Within the whole Archimedean class, though, there exist different
copulas having the same zero set \citep[cf.][p.~132]{Nel:06}, such
that this criterion fails.
\end{example}
\noindent Theorem~\ref{res:one-parameter} has a corollary that provides us with
an identifiability criterion in terms of the generator.

\begin{coro}\label{prop:critztheta} 
Let $\Phi:=\{\varphi_\theta \;|\; \theta\in\Theta\subset\R\}\subset\Omega$ be
a family of bounded generators such that $\theta\mapsto \varphi_\theta(t)/\varphi_\theta(0)$
is strictly monotonic in $\theta$ for every $t\in(0,1)$. Suppose that
the marginal distributions $F_X$ and $F_Y$ are arbitrary but known.
Then, the copulas generated by the class $\Phi$ are identifiable.
\end{coro}
\begin{proof}{}{}
  Recall that a constant times a generator results in the same
  Archimedean copula. Thus, without loss of generality, we can 
  consider the following class of generators instead of $\Phi$ which generates the same
  copulas:
  $\tilde\Phi=\{\tilde\varphi_\theta:=\varphi_\theta/\varphi_\theta(0)
  \;|\; \theta\in\Theta\subset\R\}\subset\Omega$.  
Note that by construction  $\tilde\varphi_\theta(0)=1$ for all~$\theta$ and $\tilde\varphi_\theta(t)$ is strictly monotonic as a
function of $\theta$ for all $t\in (0,1)$ (say strictly increasing, without loss of generality).

  Consider the intersection point $\zeta^\theta$ of the copula's zero curve
  with the curve $\Psi$ defined in~\eqref{eq:curve_Psi}. This
  point can be identified by means of the distribution function~$H$
  (cf. Theorem~\ref{res:one-parameter}). We show below that
  for $\theta'>\theta$, the zero curve of $C_{\theta}$ lies strictly
  above the one of $C_{\theta'}$ on $(0,1)$. This implies that two
  different choices of $\theta$ always result in two different
  intersection points $\zeta$ and thus completes the proof.

 The zero curve of the copula generated by $\tilde\varphi$ is
  given by $t\mapsto \tilde\varphi^{-1}(1 - \tilde\varphi(t))$ for
  $t\in[0,1]$.
Let $t\in(0,1)$. We have that (using Lemma~\ref{lem:inverse} in the
penultimate step)

\newcommand{\ftp}{\tilde\varphi_{\theta'}}
\newcommand{\ft}{\tilde\varphi_{\theta}}

\begin{alignat*}{2}
&&\ftp(t) &> \ft(t)  \\
&\iff&  \ftp(t) &> \ft(t) - \underbrace{(\ft(0) - \ftp(0))}_{=(1-1)=0}
\\
&\iff &  \ft(0) - \ft(t)  &> \ftp(\ftp^{-1}(\ftp(0)-\ftp(t)))\\
&\implies&   \ft(0) - \ft(t)  &> \ft(\ftp^{-1}(\ftp(0)-\ftp(t)))\\
&\iff& \ft^{-1}(\ft(0) - \ft(t)) &< \ftp^{-1}(\ftp(0) - \ftp(t)).
\end{alignat*}
This completes the proof. 
\end{proof}

\begin{example} 
\label{ex:corollary}
The class (4.2.2) from \cite{Nel:06}, which is given by
$\Phi:=\{\varphi_\theta(t) = (1-t)^\theta\;|\; \theta\in[1,\infty)\}$,
is identifiable by virtue of Corollary~\ref{prop:critztheta}, because $\varphi_\theta(0)
= 1$ for all $\theta$, and $\varphi_\theta(t)$ is strictly decreasing
in $\theta$ for all $t\in(0,1)$.  In fact,
Corollary~\ref{prop:critztheta} can be applied successfully to all
classes of bounded generators listed in \citet[ Table~4.1]{Nel:06}.  As
for most of the remaining classes (strict or not),
Theorem~\ref{res:one-parameter} 
can be applied successfully. The required computations become rather
unwieldy, though.
\end{example}


\subsection{Asymmetric copulas}
\label{sec:asymmetric-copulas}

So far we have restricted our attention to families of symmetric
copulas, which contain some of the most widely used classes. In this
section, we discuss briefly some families of asymmetric
copulas. Let us consider a copula density of the form
$c(x,y) = \tc(x-y)$ with a 1-periodic function $\tc:\R\to
[0,\infty)$ such that $\int_0^1\tc(x)\,dx=1$.  \cite{AB:05}
show that this density actually yields a
copula $C$ which they call a \textit{periodic copula}. Note that such
a periodic copula can be asymmetric (non-exchangeable) when $\tc$ is
not an even function. Let us consider two examples in which classes of
asymmetric copulas are identifiable in the context of the competing
risks model in the special case where the margins are uniform. The
construction of these classes is according to \cite{AB:05}.

\begin{example}[Periodic jump copulas] Suppose that $X$ and $Y$ are
  uniformly distributed on the unit interval.  For $\gamma\in(0,1/2)$,
  let $\tc_\gamma$ be the periodic continuation on $\R$ of
  $\gamma^{-1}\1_{[0,\gamma]}(x)$ (with $0\leq x<1$), denote by
  $C_\gamma$ the corresponding periodic copula, and let $\cC :=
  \{C_\gamma\;|\;\gamma\in(0,1/2)\}$.  It is easy to see
  that
  \[C_{x,\gamma}(z) := \frac{\partial}{\partial x}C(x,z)\Big|_{x=z}
  = \frac{z}{\gamma}\wedge 1.\] This implies that for two different
  $\gamma,\gamma'\in(0,1/2)$, we have that $C_{x,\gamma}(\gamma\wedge\gamma')\neq
  C_{x,\gamma'}(\gamma\wedge\gamma')$, whence the identifiability of
  $\cC$.
\end{example}
There are also classes of asymmetric copulas that are not
identifiable. We end this section with such a negative example.

\begin{example}[Generalised Cuadras-Aug\'e family]
  This class is defined  by
  \[C_{\alpha,\beta}(u,v) =
  \begin{cases}
    u^{1-\alpha}v & u^\alpha\geq v^\beta\\
    uv^{1-\beta} & u^\alpha\leq v^\beta.\\
  \end{cases}
  \]
  for $0<\alpha,\beta<1$ \citep[cf.][p.52ff]{Nel:06}. 
\newcommand{\fab}{f_{\alpha,\beta}}
\newcommand{\Cab}{C_{\alpha,\beta}}
The domain of such a copula is divided into two parts by the graph of
the function  $\fab(x)= x^{\alpha/\beta}$. In the upper left part of
the domain, the copula only depends on the parameter $\beta$, in the
lower right part only on $\alpha$. Thus, the parameter $\alpha$
is obviously determined by the value of the copula at any single point in the lower right part of
the domain, and the parameter $\beta$ by such a value in the other
part of the domain.

Recall that by virtue of~\eqref{eq:jointH}, the copula's values along the curve $\Psi$ defined
in~\eqref{eq:curve_Psi} are identifiable based on the distribution of
$(Z, \Delta)$. In view of the above discussion this means that a
copula's parameters $\alpha$ and $\beta$ can be identified if the
curve $\Psi$ crosses the graph of the function $\fab$ and, more
precisely, there are two points on $\Psi$ of which we know that they
lie on opposite sides of the graph of $\fab$. Observe that the curve~$\Psi$ is
the graph of the function
\[
\psi:[0,1]\to [0,1]:\quad u \mapsto F_Y(F_X^{-1}(u))
\]
and suppose that the marginal distributions $F_X$ and $F_Y$ are such
that the expression $\log\psi(u) / \log u$ is nowhere constant as a
function of $u$. The latter condition amounts to assuming that there
is no constant $r>0$ such that $F_X(t) = F_Y(t)^r$ on any interval
inside $[0,1]$. 

Let $z\in\R$ and $u:=u(z):=F_X(z)$. Then, we have that
$H(z)=C(u,\psi(u))$ and, by definition of the Cuadras-Aug\'e copulas,
\[H(z) = 
  \begin{cases}
    u^{1-\alpha}\psi(u) & u^\alpha\geq \psi(u)^\beta\\
    u\psi(u)^{1-\beta} & u^\alpha\leq \psi(u)^\beta.\\
  \end{cases}
\] 
Dividing by $u\psi(u)$, taking the logarithm, and finally dividing by
$\log u$, we obtain
\[\Gamma(z):=\frac{\log H(z) - \log u -\log \psi(u)}{\log u} = 
  \begin{cases}
    -\alpha & \psi(u)\leq f_{\alpha,\beta}(u)\\
    -\beta\frac{\log \psi(u)}{\log u} & \psi(u)\geq f_{\alpha,\beta}(u).\\
  \end{cases}
\] 
The left hand side of the last equation is identifiable based on the distribution of $(Z,
\Delta)$ and we may therefore consider it as known. The right hand side is constant for $\psi(u)\leq
f_{\alpha,\beta}(u)$ and nowhere constant otherwise by
assumption. Suppose that the curve $\Psi$ lies on both sides of the graph
of $f_{\alpha,\beta}$. On the one hand, function $\Gamma$ is then constant on some
intervals where it takes the value $-\alpha$. On the other hand,
there are intervals where $\Gamma$ is nowhere constant and
$\Gamma(z)\log (u)/\log(\psi(u)) = -\beta$. As $\log(u)/\log(\psi(u))$
is known, both parameters $\alpha$ and $\beta$ are identifiable in
this case. To summarise, we obtain the following sufficient
identifiability condition:

A set of parameters  $\mathcal{A} \subset
(0,1)^2\setminus\{(t,t)\;|\;t\in(0,1)\}$  defining a subclass of the
generalised Cuadras-Aug\'e family is identifiable if 
\[
\forall\; (\alpha,\beta)\in\mathcal{A}\quad\exists\; x_1,x_2\in(0,1)
\;:\; \psi(x_1) > \fab(x_1) \;\text{ and }\; \psi(x_2) < \fab(x_2)
\]
and if additionally $\log \psi(u) / \log u$ is nowhere constant as a function of $u$.

If the latter condition is not satisfied, but we know two points on
$\Psi$ that lie on opposite sides of the graph of $f_{\alpha,\beta}$ for every
$(\alpha,\beta)\in\mathcal{A}$, then the parameters are still
identified, which yields the following condition:
\[
\exists\; x_1,x_2\in(0,1)\quad\forall\; (\alpha,\beta)\in\mathcal{A}
\;:\; \psi(x_1) > \fab(x_1) \;\text{ and }\; \psi(x_2) < \fab(x_2).
\]
\end{example}


\newpage
\section{One unknown margin}
\label{sec:unknown-FX}

In this section, we consider the case where one of the
two margins is unknown. Our objective is to show that the unknown
margin $F_X$ and the copula $C_{XY}$ can be identified simultaneously
under certain conditions. To this end, we need stronger assumptions
than in the previous section, where both margins were entirely
known. More specifically, we will suppose that the two observed
variables 
\begin{equation}
\Delta \text{ and } Z \text{ are stochastically independent.} \label{ass:KG}
\end{equation}
As far as the family of copulas is concerned, we  use the
same conditions as in Corollary~\ref{prop:critztheta}. We show below that in this
setting, the copula is identifiable.

It is noteworthy that in the context of classical survival analysis,
i.e. assuming independence of the competing risks $X$ and $Y$, the
additional condition~\eqref{ass:KG} is equivalent to assuming the
so-called Koziol-Green model. In this censoring model, one assumes
that the survival function of the event time $X$ is a power of the
survival function of the censoring time $Y$. The Koziol-Green model
has been applied successfully in some practical situations. For
example, \cite{KG:76} and \cite{CH:81} consider prostate cancer data,
whereas \cite{Cso:88} treats the Channing House data from
\cite{Hyd:77}. Nevertheless, the classical Koziol-Green model has been
criticised as unrealistic in many settings, e.g. \cite{CF:98} called it
``too good to be frequently true''. Numerous modifications of the
original model have been suggested, e.g. the Generalised
Koziol-Green model in which hypothesis~\eqref{ass:KG} is weakened by
assuming that $\Delta$ and $Z$ may be dependent, but according to a
known copula. It is not within the scope of this paper to contribute to the
debate on the usefulness of the Koziol-Green model and its
modifications. Our purpose is rather the extension of the
identification results of the previous section to the context of
dependent competing risks with one unknown margin.

When $X$ and $Y$ are allowed to be dependent, condition~(\ref{ass:KG})
is weaker than the original ``proportional hazards'' assumption made
by \cite{KG:76}. \cite{BV:08} consider this case and develop a
strongly consistent estimator of $F_X$ in this context, even when
$F_Y$ is unknown. However, they assume the copula~$C_{XY}$ to be
known.  The next result shows that under assumption~(\ref{ass:KG}),
the distribution function~$F_X$ and the copula $C_{XY}$ can be
identified simultaneously (the latter within certain parametric
classes).

\begin{theo}\label{thm:koziol_green_identify}
  Suppose that $\Delta$ and $Z$ are independent and that $F_Y$ is
  known. Assume further that both $F_X$ and $F_Y$ are continuous and
  strictly increasing in $(0,\infty)$.  As in
  Corollary~\ref{prop:critztheta}, let $\Phi:=\{\varphi_\theta \;|\;
  \theta\in\Theta\subset\R\}\subset\Omega$ be a family of bounded
  generators such that $\theta\mapsto
  \varphi_\theta(t)/\varphi_\theta(0)$ is strictly monotonic in
  $\theta$ for every $t\in(0,1)$.  Then, the unknown marginal law
  $F_X$ and the copulas generated by the class $\Phi$ are
  identifiable.
\end{theo}

\newpage
\begin{proof}{}{}
  For given $z \in (0,1)$, the level curve of level $z$ intersects
  with the curve $\Psi=(F_X(t),F_Y(t))_{t\in\R}$ for
$$ t = F_X^{-1} \Big[ \varphi_\theta^{[-1]} \Big\{- \int_{H^{-1}(z)}^1 \varphi_\theta^\prime(H(s)) \, dH(s,1) \Big\} \Big] $$
or equivalently at the point
\begin{eqnarray} \label{copul}
  && \Big(\varphi_\theta^{[-1]} \Big\{- \int_{H^{-1}(z)}^1 \varphi_\theta^\prime(H(s)) \, dH(s,1) \Big\}, \varphi_\theta^{[-1]} \Big\{- \int_{H^{-1}(z)}^1 \varphi_\theta^\prime(H(s)) \, dH(s,0) \Big\}\Big) \nonumber \\
  && := (\zeta^\theta_x(z), \zeta^\theta_y(z)).
\end{eqnarray}
In order to prove (\ref{copul}), we need to show that
$C(\zeta^\theta_x(z), \zeta^\theta_y(z)) = z$.  Since $C(x,y) =
\varphi_\theta^{[-1]}(\varphi_\theta(x) + \varphi_\theta(y))$, this means that we need to
show that $\varphi_\theta(\zeta^\theta_x(z)) + \varphi_\theta(\zeta^\theta_y(z)) = \varphi_\theta(z)$.
Indeed,
\begin{eqnarray*}
  && \varphi_\theta(\zeta^\theta_x(z)) + \varphi_\theta(\zeta^\theta_y(z)) \\
  && =  - \int_{H^{-1}(z)}^1 \varphi_\theta^\prime(H(s)) \, dH(s,1) -  \int_{H^{-1}(z)}^1 \varphi_\theta^\prime(H(s)) \, dH(s,0) \\
  && =  - \int_{H^{-1}(z)}^1 \varphi_\theta^\prime(H(s)) \, dH(s) \\
  && = - \int_z^1 \varphi_\theta^\prime(u) \, du = \varphi_\theta(z),
\end{eqnarray*}
since $\varphi_\theta(1)=0$. The independence of $\Delta$ and $Z$ implies
that there is a constant $\alpha \in [0,1]$ such that
\begin{align*}
  H(s,0) &= \alpha H(s) \\
  H(s,1) &= (1-\alpha) H(s),
\end{align*}
and thus, using~\eqref{copul}, we can write $\zeta^\theta_y(z) =
\varphi_\theta^{[-1]}(\alpha\varphi_\theta(z))$.  Following the proof
of Corollary~\ref{prop:critztheta}, we obtain that
$\zeta^\theta_y(0)$ is strictly monotonic in $\theta$ under the
assumptions.

Now let $\xi := \sup\{t\in[0,1]\;|\; C(F_X(t), F_Y(t)) =
0\}$. Note that in view
of~\eqref{eq:6}, $\xi$ is identifiable based on the
distribution of $(Z, \Delta)$. Furthermore, we have $\zeta^\theta_y(0) = F_Y(\xi)$,
which then implies the identifiability of $\theta$ using the monotonicity of
$\zeta^\theta_y(0)$ in~$\theta$. 
Having identified $\theta$, we may consider the copula as known. The
result then follows applying Theorem~3.1 from \cite{ZK:95}.
\end{proof}


\section{Discussion}
\label{sec:conclusion}

We have considered the problem of identifiability in a bivariate
competing risks model with dependence between the two event times $X$
and $Y$. It is well known that without further restrictions on the
model, the joint distribution of the event times is not identifiable
when only the minimum $Z$ of the two times and the indicator~$\Delta$
are observed.

First, we showed as an intermediate step that when the marginal
distributions of $X$ and $Y$ are known, their copula is identifiable
within many popular families of Archimedean copulas, but also in some
less known families. However, in many applications at least one of the
marginal distributions is the object of interest that has to be
estimated.

In the case where only one of the two margins is known, we were still
able to establish an identification result, but we needed the
additional assumption that $Z$ and $\Delta$ are stochastically
independent. Although this hypothesis can be tested in practice, it
would be desirable to replace it by a weaker assumption. One could
imagine that $Z$ and $\Delta$ are dependent via some copula
$C_{Z\Delta}$, but it is not obvious in which way this copula has to
be related to the copula $C_{XY}$ and how identifiability can be shown
in this setting.

The problem becomes still harder when both margins are completely
unknown.  Without any further assumptions on the margins, the class of
all Archimedean copulas is too large as to allow for identification
\citep{Wan:12}, and at present, we do not know if identification is
possible for certain subfamilies of Archimedean copulas. One could try
to tackle this problem assuming the margins to lie in parametric
classes. 

The generalisation of the results in this paper to the multivariate
case should be straightforward, at least for the Archimedean families
of copulas. Further research is needed for the development of
estimation methods in the models that we have shown to be
identifiable.


\appendix

\section{Appendix}
\label{sec:appendix}

\begin{lem}\label{lem:inverse}
  If $\{f_\theta\}$ is a family of monotonic
  bijective functions $f_\theta:[0,1]\to[0,1]$ such that $f_\theta(t)$
  is strictly increasing in $\theta$ for every $t\in (0,1)$, then
  $f^{-1}_\theta(t)$ is also strictly increasing in~$\theta$ for every
  $t\in(0,1)$.
\end{lem}
\begin{proof}{}{}
For $\theta'>\theta$ and a given $t\in(0,1)$,
  let $z:=f_{\theta'}^{-1}(t)$. Then, we have $t=f_{\theta'}(z)$, and consequently
  \begin{equation}
    f^{-1}_{\theta'}(t) > f^{-1}_\theta(t) \iff z >
    f^{-1}_{\theta}(f_{\theta'}(z)) \iff f_\theta(z) < f_{\theta'}(z),\label{eq:z0}
  \end{equation}
  which is true by assumption (see Figure~\ref{fig:phi} for an
  illustration of the second inequality).
  \begin{figure}[h]
    \centering
    \includegraphics[scale=.5]{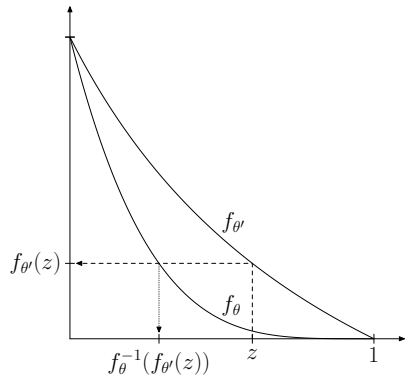}
    \caption{Illustration of inequality~\eqref{eq:z0}}
    \label{fig:phi}
  \end{figure}
\end{proof}


\section*{Acknowledgments}
\label{sec:acknowledgments}

  I. Van Keilegom and M. Schwarz acknowledge financial support from
  the IAP research network P7/06 of the Belgian Government (Belgian
  Science Policy) and from the European Research Council under the
  European Community's Seventh Framework Programme (FP7/2007-2013) /
  ERC Grant agreement No. 203650. I. Van Keilegom acknowledges further
  financial support from the contract ``Projet d'Actions de Recherche
  Concert\'ees'' (ARC) 11/16-039 of the ``Communaut\'e fran\c{c}aise de
  Belgique'', granted by the ``Acad\'emie universitaire Louvain''.  

\bibliographystyle{apalike}
\bibliography{dr}

\end{document}